\documentclass[a4paper,11pt]{article}

\usepackage{amsmath}
\usepackage{amssymb}
\usepackage{amsthm}
\usepackage{graphicx}
\usepackage{enumerate}
\newtheorem{theorem}{Theorem}[section]

\newtheorem{lemma}[theorem]{Lemma}

\theoremstyle{definition}

\newtheorem{remark}[theorem]{Remark}

\begin{document}
	
\title{A Discrete Multi-Sequence Cauchy-Schwarz-Like Inequality}
\author{Nihal Uppugunduri}
\date{}
\maketitle

\begin{abstract}
We prove a general inequality for more than two sequences mirroring that of the discrete two-sequence Cauchy-Schwarz.
\end{abstract}

\section{Introduction}

The ``traditional" Cauchy-Schwarz inequality for real numbers states that, if $a_i,b_i\in\mathbb{R}$ for $i\in [1,n]$, then
$$\sum_{i=1}^n {a_i^2} \sum_{i=1}^n {b_i^2} \geq \left(\sum_{i=1}^n {a_i b_i}\right)^2,$$
with equality iff one of the sequences is all zeros or there exists $c\neq 0$ such that $a_i=cb_i$ for all $i$. Many extensions and generalizations to this result have been established over the years, as discussed in \cite{dragomir} and \cite{yin}.

In this article, we prove an inequality that mirrors Cauchy-Schwarz for more than two sequences that has not yet seemed to appear in the literature. Specifically, we show that, if $a_{ij}\in\mathbb{R}$ for $i\in\ [1,n]$ and $j\in [1,m]$, then
$$\prod_{j=1}^m {\sum_{i=1}^n {a_{ij}^2}} \geq \left(\sum_{i=1}^n{\prod_{j=1}^m{a_{ij}}}\right)^2,$$
and we determine when equality can occur. We also derive a corresponding inequality over the complex numbers and determine its equality condition.

\section{Results}

We begin with an immediate lemma.

\begin{lemma}\label{main_lemma}
    Consider $\{c_i\}_{i=1}^n$ with $c_i\geq 0$. The second elementary symmetric sum $\sum_{1\leq i<j\leq n}{c_ic_j}$
    is zero iff at most one of the $c_i$ is nonzero.
\end{lemma}
\begin{proof}
    Since $c_ic_j\geq 0$ but their sum is zero, $c_ic_j=0$ for all $i\neq j$. WLOG assume $c_1\neq 0$. Then $c_1c_i=0\rightarrow c_i=0, i\neq 1$. Conversely, since each term in the sum is $c_ic_j$ where $1\leq i<j$, $c_j=0\rightarrow c_ic_j=0$ and the sum vanishes.
\end{proof}

\begin{theorem}\label{main_result}
    If $a_{ij}\in\mathbb{R}$ for $i\in\ [1,n]$ and $j\in [1,m]$, with $m\geq 3$, then
    $$\prod_{j=1}^m {\sum_{i=1}^n {a_{ij}^2}} \geq \left(\sum_{i=1}^n{\prod_{j=1}^m{a_{ij}}}\right)^2,$$
    with equality iff either:
    \begin{enumerate}
        \item There exists $j'$ such that $a_{ij'}=0$, so that one of the sequences is all zeros, or
        \item There exists $i'$ such that $a_{ij}=0$ for all $i\neq i'$, so that there is at most one nonzero list of corresponding numbers in the sequences.
    \end{enumerate}
\end{theorem}

\begin{proof}
	We first consider only $a_{ij}\geq 0$. To establish the inequality, we proceed by induction on $m\geq 2$, where the base case is the original Cauchy-Schwarz:
	\begin{align*}
	    \prod_{j=1}^m {\sum_{i=1}^n {a_{ij}^2}} &= \left(\prod_{j=1}^{m-1} {\sum_{i=1}^n {a_{ij}^2}}\right)\sum_{i=1}^n{a_{im}^2} \\
	    &\geq \left(\sum_{i=1}^n{\prod_{j=1}^{m-1}{a_{ij}}}\right)^2 \sum_{i=1}^n{a_{im}^2} \\
	    &\geq \sum_{i=1}^n{\left(\prod_{j=1}^{m-1}{a_{ij}}\right)^2} \sum_{i=1}^n{a_{im}^2} \\
	    &\geq \left(\sum_{i=1}^n{\prod_{j=1}^m{a_{ij}}}\right)^2.
	\end{align*}
	Next, we establish the necessity of the equality condition for $m\geq 3$ by induction. Assume that equality condition 1 is not the case. For the base case $m=3$, equality is achieved only if
	$$\left(\sum_{i=1}^n{a_{i1}a_{i2}}\right)^2 = \sum_{i=1}^n{\left(a_{i1}a_{i2}\right)^2},$$
	which occurs exactly when the second elementary symmetric sum of $\{a_{i1}a_{i2}\}_{i=1}^n$ vanishes. Thus, by \ref{main_lemma}, there exists $i'$ such that $a_{i1}a_{i2}=0$ for $i\neq i'$. But equality is achieved only if
	$$\sum_{i=1}^n{a_{i1}^2}\sum_{i=1}^n{a_{i2}^2} = \left(\sum_{i=1}^n{a_{i1}a_{i2}}\right)^2,$$
	so since equality condition 1 is not true, there exists $c\neq 0$ such that $a_{i2}=ca_{i1}$ for all $i$. Thus, for $i\neq i'$, $a_{i1}a_{i2}=ca_{i1}^2=0\rightarrow a_{i1}=0$ and $a_{i2}=c0=0$. Additionally, equality is achieved only if
	$$\sum_{i=1}^n{\left(a_{i1}a_{i2}\right)^2} \sum_{i=1}^n{a_{i3}^2} = \left(\sum_{i=1}^n{a_{i1}a_{i2}a_{i3}}\right)^2,$$
	so since equality condition 1 is not true, there exists $c\neq 0$ such that $a_{i3}=ca_{i1}a_{i2}$ for all $i$. Then, for $i\neq i'$, $a_{i3}=c0=0$, and equality condition 2 is established.
	
	If $m>3$, then equality is achieved only if
	$$\prod_{j=1}^{m-1} {\sum_{i=1}^n {a_{ij}^2}}
	= \left(\sum_{i=1}^n{\prod_{j=1}^{m-1}{a_{ij}}}\right)^2,$$
	so by inductive hypothesis there exists $i'$ such that $a_{ij}=0$ for $i\neq i'$ and $j<m$. Additionally, equality is achieved only if
	$$\sum_{i=1}^n{\left(\prod_{j=1}^{m-1}{a_{ij}}\right)^2} \sum_{i=1}^n{a_{im}^2} =
	\left(\sum_{i=1}^n{\prod_{j=1}^m{a_{ij}}}\right)^2,$$
	so since equality condition 1 is not true, there exists $c\neq 0$ such that $a_{im}=c\prod_{j=1}^{m-1}{a_{ij}}$ for all $i$. Then, for $i\neq i'$, $a_{im}=c0=0$, and equality condition 2 is established.
	
	Now, take arbitrary $a_{ij}\in\mathbb{R}$. We have
	\begin{align*}
	    \prod_{j=1}^m {\sum_{i=1}^n {a_{ij}^2}} &= \prod_{j=1}^m {\sum_{i=1}^n {|a_{ij}|^2}} \\
	    &\geq \left(\sum_{i=1}^n{\prod_{j=1}^m{|a_{ij}|}}\right)^2 \\
	    &\geq \left(\sum_{i=1}^n{\prod_{j=1}^m{a_{ij}}}\right)^2,
	\end{align*}
	with equality only if there exists $j'$ such that $|a_{ij'}|=0 \rightarrow a_{ij'}=0$, which is equality condition 1, or there exists $i'$ such that $|a_{ij}|=0 \rightarrow a_{ij}=0$ for all $i\neq i'$, which is equality condition 2. Furthermore, either of these conditions is sufficient.
\end{proof}

\begin{remark}
    The corresponding integral inequality does not hold. Specifically, for any $m,n\in\mathbb{Z}_+$ with $m>2$ and any $R\subseteq \mathbb{R}^n$ with nonzero measure, it is not the case that for all functions $\{f_i\}_{i=1}^m$, $f_i:R\rightarrow\mathbb{R}$, we have
    $$\prod_{i=1}^m\left(\int_R (f_i(x))^2 dx\right) \geq \left(\int_R \left(\prod_{i=1}^m f_i(x)\right) dx\right)^2,$$
    when all the integrals exist. A direct counterexample results from letting $f_i=c$ on a subset of $R$ with nonzero measure less than one, and zero everywhere else on $R$.
\end{remark}

\begin{remark}
    We have an analogous result over the complex numbers: if $a_{ij}\in\mathbb{C}$ for $i\in\ [1,n]$ and $j\in [1,m]$, with $m\geq 3$, then
    $$\prod_{j=1}^m {\sum_{i=1}^n {|a_{ij}|^2}} \geq \left|\sum_{i=1}^n{\prod_{j=1}^m{a_{ij}}}\right|^2,$$
    with equality condition same as above. This follows from \ref{main_result} and the generalized triangle inequality:
    $$\prod_{j=1}^m {\sum_{i=1}^n {|a_{ij}|^2}} \geq \left(\sum_{i=1}^n{\prod_{j=1}^m{|a_{ij}|}}\right)^2
    \geq \left|\sum_{i=1}^n{\prod_{j=1}^m{a_{ij}}}\right|^2.$$
    The equality condition for \ref{main_result} implies the necessity of the complex number condition, and the condition is still sufficient.
\end{remark}

\begin{remark}
The corresponding integral inequality again does not hold for complex numbers: for any $m,n\in\mathbb{Z}_+$ with $m>2$ and any $R\subseteq \mathbb{R}^n$ with nonzero measure, it is not the case that for all functions $\{f_i\}_{i=1}^m$, $f_i:R\rightarrow\mathbb{C}$, we have
$$\prod_{i=1}^m\left(\int_R |f_i(x)|^2 dx\right) \geq \left|\int_R \left(\prod_{i=1}^m f_i(x)\right) dx\right|^2,$$
when all the integrals exist. A direct counterexample results from letting $f_i=c\in\mathbb{R}$ on a subset of $R$ with nonzero measure less than one, and zero everywhere else on $R$.
\end{remark}

\end{document}